\documentclass[a4paper, 10pt, oneside]{amsart}
\usepackage{algorithm, algorithmic, microtype}
\usepackage[textwidth=16cm, textheight=23.4cm]{geometry}

\newtheorem{theorem}{Theorem}[section]

\newtheorem{corollary}[theorem]{Corollary}

\newtheorem*{conjecture}{Conjecture}

\theoremstyle{definition}

\newtheorem{remark}[theorem]{Remark}

\newcommand{\N}{\mathbb N}

 \DeclareMathOperator{\ord}{ord}

\renewcommand{\t}{\, | \,}


\begin{document}
\address{Institut f\"ur Mathematik und Wissenschaftliches Rechnen \\Karl-Franzens-Universit\"at Graz \\Hein\-rich\-stra\ss e 36\\8010 Graz, Austria}
\email{alfred.geroldinger@uni-graz.at, manfred.liebmann@uni-graz.at, andreas.philipp@uni-graz.at}
\author{Alfred Geroldinger and Manfred Liebmann and Andreas Philipp}
\thanks{This work was supported by the Austrian Science Fund FWF, Project No. P21576-N18.\\ \indent
We kindly acknowledge the support of the DECI (Distributed Extreme Computing Initiative) within the muHEART\\ \indent project for providing access to the \emph{cineca} supercomputer.}
\title{On the Davenport constant and on the structure of extremal zero-sum free sequences}

\begin{abstract}
Let $G = C_{n_1} \oplus \ldots \oplus C_{n_r}$ with $1 < n_1 \t
\ldots \t n_r$ be a finite abelian group, $\mathsf d^* (G) = n_1 +
\ldots + n_r - r$, and let $\mathsf d (G)$ denote the maximal length
of a zero-sum free sequence over $G$. Then $\mathsf d (G) \ge
\mathsf d^* (G)$, and the standing conjecture is that equality holds
for $G = C_n^r$. We show that equality does not hold for $C_2 \oplus
C_{2n}^r$, where $n \ge 3$ is odd and $r \ge 4$. This gives new
information on the structure of extremal zero-sum free sequences
over $C_{2n}^r$.
\end{abstract}

\maketitle

\bigskip
\section{Introduction} \label{1}
\bigskip

Let $G$ be an additively written finite abelian group, $G = C_{n_1}
\oplus \ldots \oplus C_{n_r}$ its direct decomposition into cyclic
groups, where $r = \mathsf r (G)$ is the rank of $G$ and $1 < n_1 \t
\ldots \t n_r$, and set
\[
\mathsf d^* (G) = \sum_{i=1}^r (n_i-1) \,, \quad \text{with} \quad
\mathsf d^* (G) = 0 \quad  \text{for} \ G \ \text{trivial} \,.
\]
We denote by $\mathsf d (G)$ the maximal length of a zero-sum free
sequence over $G$. Then $\mathsf D (G) = \mathsf d (G)+1$ is the
Davenport constant of $G$ (equivalently, $\mathsf D (G)$ is the
smallest integer $\ell \in \N$ such that every sequence $S$ over $G$
of length $|S| \ge \ell$ has a non-trivial zero-sum subsequence). The
Davenport constant has been studied since the 1960s, and it
naturally occurs in various branches of combinatorics, number
theory, and geometry. There is a well-known chain of inequalities
\[
\mathsf d^* (G) \le \mathsf d (G)  \le (n_r-1) + n_r \log \frac{|G|}{n_r} \,, \tag{$*$}
\]
which obviously is an equality for cyclic groups (\cite[Theorem 5.5.5]{Ge-HK06a}). Furthermore, equality on the left side holds for $p$-groups, groups of rank  two and
others (see \cite[Sections 2.2 and 4.2]{Ge09a} for a survey, and
\cite{Bh-SP07a, Su09a, Bh-Ha-SP09a, Yu-Ze09c, Fr-Sc10a, Sm11a, O-P-S-S12} for recent
progress). In contrast to these results, there are only a handful of explicit families of examples showing that $\mathsf d (G) > \mathsf d^* (G)$ can
happen, but the phenomenon is not understood at all. The two main
conjectures regarding $\mathsf D(G)$ state that  equality holds in the left side of $(*)$ for groups of rank
three and for groups of the form $C_n^r$.

In addition to the direct problem, the associated inverse problem
with respect to the Davenport constant---which asks for the
structure of maximal zero-sum free sequences---has attracted
considerable attention in the last decade. An easy exercise shows
that a zero-sum free sequence of maximal length over a cyclic group
consists of one element with multiplicity $\mathsf d (G)$. A
conjecture on the structure of such sequences over groups of the
form $C_n \oplus C_n$ was first stated in \cite[Section
10]{Ga-Ge99}. After various partial results, this conjecture was
settled recently: even  for general groups of rank two  the
structure of minimal zero-sum sequences with maximal length    was
completely determined (see \cite{Ga-Ge-Gr10a, Sc10b, Re11a}). Apart
from groups of rank two (and apart from the trivial case of
elementary $2$-groups) such a structural result is known only for
groups of the form $C_2^2 \oplus C_{2n}$ (see \cite{Sc10c}).

The inverse
results for groups of rank two support the conjecture that $\mathsf d^* (G) = \mathsf d (G)$ holds for groups of rank
three (which is outlined in \cite{Sc10c}). Much less is known for
groups of the form $C_n^r$. There is a covering result
(\cite[Theorem 6.6]{Ga-Ge03a}), which slightly supports the
conjecture that $\mathsf d^* (G) = \mathsf d (G)$ holds, and there is recent work by B. Girard
(\cite{Gi08b, Gi10a}) on the order of elements occurring in zero-sum
free sequences of maximal length.

In this paper, we present a series of groups of rank five, namely
$G_n = C_2 \oplus C_{2n}^4$ with $n \ge 3$ odd, such that $\mathsf d
(G_n)
> \mathsf d^* (G_n)$ (see Theorem \ref{3.1}). This is the first
series of groups for which equality in the left side of $(*)$ fails and which is somehow close to
the form $C_n^r$ (all groups known so far satisfying $\mathsf d^* (G) < \mathsf d (G)$  are
quite different). Moreover, these examples shed new light on recent
conjectures by B. Girard concerning the structure of extremal
sequences (see Corollary \ref{3.2} and the subsequent remark). A
computer based search in the group $C_2 \oplus C_{10}^4$ was
substantial for our work. This will be outlined in Section \ref{4}.

\bigskip
\section{Preliminaries} \label{2}
\bigskip

Our notation and terminology are consistent with \cite{Ga-Ge06b} and
\cite{Ge-HK06a}. We briefly gather some key notions and fix the
notation concerning sequences over finite abelian groups. Let
$\mathbb N$ denote the set of positive integers, $\mathbb P \subset
\N$ the set of prime numbers, and let $\mathbb N_0 = \mathbb N \cup
\{ 0 \}$. For  $a, b \in \mathbb Z$, we set $[a, b] = \{ x \in
\mathbb Z \mid a \le x \le b\}$.    Throughout, all abelian
groups will be written additively, and for $n \in \mathbb N$, we denote by  $C_n$  a cyclic group with $n$ elements.

Let $G$ be a finite abelian group. For a subset $A \subset G$, we set  $-A = \{-a \mid a \in A\}$.
An $s$-tuple $(e_1, \ldots, e_s)$ of elements of $G$ is said to be
 independent (or more briefly, the elements $e_1, \ldots, e_s$ are said
to be independent) \ if $e_i \ne 0$ for all $i \in [1,s]$ and, for
every $s$-tuple $(m_1, \ldots, m_s) \in \mathbb Z^{s}$,
\[
m_1 e_1 + \ldots + m_s e_s  =0 \qquad \text{implies} \qquad m_1e_1 =
\ldots = m_se_s = 0 \,.
\]
An $s$-tuple $(e_1, \ldots, e_s)$ of elements of $G$ is called a \
basis \ if it is independent and $G = \langle e_1\rangle \oplus
\ldots \oplus \langle e_s \rangle$. For a prime $p \in \mathbb P$,
we denote by $G_p = \{g \in G \mid \ord (g) \ \text{is a power of} \
p \}$ the $p$-primary component of $G$, and by $\mathsf r_p (G)$,
the $p$-rank of $G$ (which is the rank of $G_p$).

Let $\mathcal F(G)$ be the free abelian monoid with basis $G$. The
elements of $\mathcal F(G)$ are called \ {\it sequences} \ over $G$.
We write sequences $S \in \mathcal F (G)$ in the form
\[
S =  \prod_{g \in G} g^{\mathsf v_g (S)}\,, \quad \text{with} \quad
\mathsf v_g (S) \in \mathbb N_0 \quad \text{for all} \quad g \in G
\,.
\]
We call \ $\mathsf v_g (S)$  the \ {\it multiplicity} \ of $g$ in
$S$, and we say that $S$ \ {\it contains} \ $g$ \ if \ $\mathsf v_g
(S) > 0$.   A sequence $S_1 $ is called a \ {\it subsequence} \ of
$S$ \ if \ $S_1 \, | \, S$ \ in $\mathcal F (G)$ \ (equivalently, \
$\mathsf v_g (S_1) \le \mathsf v_g (S)$ \ for all $g \in G$). If a
sequence $S \in \mathcal F(G)$ is written in the form $S = g_1 \cdot
\ldots \cdot g_l$, we tacitly assume that $l \in \mathbb N_0$ and
$g_1, \ldots, g_l \in G$.

\smallskip

For a sequence
\[
S \ = \ g_1 \cdot \ldots \cdot g_l \ = \  \prod_{g \in G} g^{\mathsf
v_g (S)} \ \in \mathcal F(G) \,,
\]
we call
\[
|S| = l = \sum_{g \in G} \mathsf v_g (S) \in \mathbb N_0 \qquad
\text{the \ {\it length} \ of \ $S$} \,,
\]
\[
\sigma (S) = \sum_{i = 1}^l g_i = \sum_{g \in G} \mathsf v_g (S) g
\in G \qquad \text{the \ {\it sum} \ of \ $S$} \,,  \ \text{and}
\]
\[
\Sigma (S) =  \Bigl\{ \sum_{i \in I} g_i \mid \emptyset \ne I
\subset [1, l] \Bigr\} \subset G  \qquad \text{the \ {\it set of
 subsums} \ of \ $S$} \,.
\]
The sequence \ $S$ \ is called
\begin{itemize}

\item a {\it zero-sum sequence} \ if \ $\sigma (S) = 0$,

\item {\it zero-sum free} \ if there is no  non-trivial zero-sum subsequence, and

\item a {\it minimal zero-sum sequence} \ if $1 \ne S$, $\sigma(S)=0$,  and every $S'|S$ with $1\leq |S'|<|S|$  is
      zero-sum free.
\end{itemize}

\bigskip
\section{The Main Theorem and  its Corollary} \label{3}
\bigskip

\medskip
\begin{theorem} \label{3.1}
Let $G = C_2^i \oplus C_{2n}^{5-i}$ with $i \in [1,4]$ and $n \ge 3$
odd. Then $\mathsf d (G) > \mathsf d^* (G)$.
\end{theorem}

\medskip
Before we start the proof of Theorem \ref{3.1}, we would like
to remark that its statement easily extends to groups of higher
rank. Indeed, let $G = C_{n_1} \oplus \ldots \oplus C_{n_r}$ with $1
< n_1 \t \ldots \t n_r$ and let $\emptyset \ne I \subset [1, r]$. If
\[
\mathsf d \bigl( \oplus_{i \in I} C_{n_i} \bigr) > \mathsf d^*
\bigl( \oplus_{i \in I} C_{n_i} \bigr) \,,
\]
then a straightforward construction shows that $\mathsf d (G) >
\mathsf d^* (G)$ (see \cite[Proposition 5.1.11]{Ge-HK06a}). Thus the
interesting groups $G$ with $\mathsf d (G) > \mathsf d^* (G)$ are
those with small rank. Recall that there is no known group $G$ of
rank three with $\mathsf d (G)
> \mathsf d^* (G)$, and there is only one series of groups $G$ of
rank four such that $\mathsf d (G) > \mathsf d^* (G)$ (see
\cite[Theorem 3]{Ge-Sc92}).

\medskip
\begin{proof}[Proof of Theorem \ref{3.1}]
For $i \in \{3,4\}$, this follows from \cite[Theorem 4]{Ge-Sc92}, and,
for $i=2$, from \cite[Theorem 3.3]{Ga-Ge99}. Suppose that $i=1$ and
let $(e_1, \ldots, e_5)$ be a basis of $G$ with $\ord (e_1) = 2$ and
$\ord (e_2) = \ldots = \ord (e_5) = 2n$. We define
\[
g_1 = e_1+e_2, \quad  g_2 = e_1+e_3, \quad  g_3  = e_1 + e_4, \quad
g_4 = e_1 + e_5 \,,
\]
\[
\begin{aligned}
g_5 & =  \phantom{e_1 + \ } \frac{3n-1}{2} e_2 + \frac{3n+1}{2} e_3
+
\frac{3n+1}{2}e_4 + \frac{3n+1}{2} e_5 \,, \\
g_6 & =  \phantom{e_1 + \ } \frac{3n-1}{2} e_2 + \frac{3n+1}{2} e_3
+
\frac{3n-1}{2}e_4 + \frac{n+1}{2} e_5 \,, \\
g_7 & =  \phantom{e_1 + \ } \frac{3n+3}{2} e_2 + \frac{n+1}{2} e_3 +
\frac{n-1}{2}e_4 + \frac{n+1}{2} e_5 \,, \\
g_8 & = \phantom{e_1 + \ } \frac{n-1}{2} e_2 + \frac{n+1}{2} e_3 +
\frac{3n+1}{2}e_4 + \frac{n-1}{2} e_5 \,, \\
g_9 & = \phantom{e_1 + \ } \frac{n-1}{2} e_2 + \frac{n+1}{2} e_3 +
\frac{n+1}{2}e_4 + \frac{n+1}{2} e_5 \,, \\
g_{10} & = \phantom{e_1 + \ } \frac{3n+1}{2} e_2 + \frac{3n+1}{2}
e_3 +
\frac{n+1}{2}e_4 + \frac{3n+1}{2} e_5 \,, \\
g_{11} & = \phantom{e_1 + \ } \frac{n+3}{2} e_2 + \frac{3n+1}{2} e_3
+
\frac{3n+1}{2}e_4 + \frac{3n-1}{2} e_5 \,, \\
g_{12} & = e_1 + \frac{n+1}{2} e_2 + \frac{n-1}{2} e_3 +
\frac{n+1}{2}e_4 + \frac{3n+1}{2} e_5 \,, \\
\end{aligned}
\]
and assert that
\[
U = g_1^{2n-2} g_2^{2n-3} g_3^{2n-2} g_4^{2n-2} g_5 g_6 g_7 g_8 g_9
g_{10} g_{11}  g_{12}
\]
is a minimal zero-sum sequence. Obviously, $U$ is a zero-sum
sequence of length $|U|  = 8n-1 = \mathsf d^* (G)+2$. Thus it
suffices to show that $S^* = g_{12}^{-1}U$ is zero-sum free. Let
\[
S = g_1^{l_1} \cdot \ldots \cdot g_{11}^{l_{11}}
\]
be a zero-sum subsequence of $g_{12}^{-1}U$, where $l_i = \mathsf
v_{g_i} (S)$ for all $i \in [1, 11]$. Thus $l_1 \in [0, 2n-2]$, $l_2
\in [0, 2n-3]$, $l_3 \in [0, 2n-2]$, $l_4 \in [0, 2n-2]$, and $l_i
\in \{0, 1\}$ for all $i \in [5, 11]$. We have to show that $|S| =
l_1 + \ldots + l_{11} = 0$.

\smallskip
Since $\sigma (S) = 0$, we obtain the following system of initial
congruences:
\begin{align}
l_1 +  l_2 + l_3 + l_4 & \equiv 0 \mod 2 \,,\\
l_1 +   \frac{3n-1}{2}l_5 + \frac{3n-1}{2} l_6  + \frac{3n+3}{2} l_7 + \frac{n-1}{2} l_8 + \frac{n-1}{2} l_9 + \frac{3n+1}{2}l_{10} + \frac{n+3}{2} l_{11} & \equiv 0 \mod 2n \,,\\
l_2 + \frac{3n+1}{2} l_5 + \frac{3n+1}{2}l_6 + \frac{n+1}{2} l_7 +
\frac{n+1}{2} l_8 + \frac{n+1}{2} l_9 + \frac{3n+1}{2} l_{10} +
\frac{3n+1}{2} l_{11} &
\equiv 0 \mod 2n \,,\\
l_3 + \frac{3n+1}{2} l_5 + \frac{3n-1}{2}l_6 + \frac{n-1}{2} l_7 +
\frac{3n+1}{2} l_8 + \frac{n+1}{2} l_9 + \frac{n+1}{2} l_{10} +
\frac{3n+1}{2} l_{11} &
\equiv 0 \mod 2n \,,\\
l_4 + \frac{3n+1}{2} l_5 + \frac{n+1}{2}l_6 + \frac{n+1}{2} l_7 +
\frac{n-1}{2} l_8 + \frac{n+1}{2} l_9 + \frac{3n+1}{2} l_{10} +
\frac{3n-1}{2} l_{11} & \equiv 0 \mod 2n\,.
\end{align}

\medskip
By subtracting equation $(2)$ from $(3)$, subtracting $(4)$ from $(3)$, and subtracting $(5)$ from $(3)$, we obtain

\begin{align}
l_1 & \equiv l_2 + l_5 + l_6+l_8+l_9+(n-1)(l_7+l_{11}) \mod 2n \,, \\
l_3 & \equiv l_2 + l_6 + l_7 + n(l_8 + l_{10}) \mod 2n \,\mbox{,and} \\
l_4 & \equiv l_2 + nl_6 + l_8 + l_{11} \mod 2n \,.
\end{align}

\medskip
Next we form a congruence modulo $2$, namely
\[
\begin{aligned}
0 & \equiv l_1+ l_2+l_3+l_4 \\
 & \equiv l_2 + l_5 + l_6 + l_8 + l_9 + \\
 & \phantom{\equiv \ \ } l_2 + \\
  & \phantom{\equiv \ \ } l_2 + l_6 + l_7 + l_8 + l_{10} + \\
   & \phantom{\equiv \ \ } l_2 + l_6 + l_8 + l_{11} \\
    & \equiv l_5 + l_6 + l_7 + l_8 + l_9 + l_{10} + l_{11} \mod 2
    \,.
\end{aligned}
\]
Therefore we get  $l_5+l_6+l_7+l_8+l_9+l_{10} + l_{11} \in
\{0,2,4,6\}$. If $l_5+l_6+l_7+l_8+l_9+l_{10} + l_{11} = 0$, then
$\sigma (S) = 0$ implies immediately that $l_1=l_2=l_3=l_4=0$ and
thus $|S| = 0$. Thus we suppose that $l_5 + \ldots + l_{11} \in
\{2,4,6\}$.

\smallskip
Adding  $(3)$ and $(5)$ and inserting $(8)$, we obtain that
\[
2l_2 + (n+1)(l_5+l_6+l_7+l_8+l_9+l_{10}+l_{11}) \equiv 0 \mod 2n \,.
\]
Thus we get that either
\[
l_5 + \ldots + l_{11} = 2 \quad \text{and hence} \quad l_2 = n-1
\]
or
\[
l_5 + \ldots + l_{11} = 4 \quad \text{and hence} \quad l_2 = n-2
\]
or
\[
l_5 + \ldots + l_{11} = 6 \quad \text{and hence} \quad l_2 \in
\{n-3, 2n-3\} \,.
\]
We distinguish these four cases.

\bigskip
\noindent
CASE 1: \,$l_5 + \ldots + l_{11} = 2$ and $l_2 = n-1$.

\smallskip
\noindent
CASE 1.1: \,$l_6=1$.

If $l_8+l_{11} = 2$, then $l_5=l_7=l_9=l_{10}=0$, $l_1=l_3=0$, and
$l_4=1$, a contradiction to $(1)$.

If $l_8+l_{11}= 0$, then $l_4 = 2n-1$, a contradiction to $l_4 \in
[0, 2n-2]$.

Thus we get $l_8 + l_{11}=1$. If $l_8=1$, then
$l_5=l_7=l_9=l_{10}=l_{11} = 0$ and $l_1=n+1$, a contradiction to
$(2)$. If $l_8 = 0$, then $l_{11}=1$, $l_5=l_7=l_9=l_{10}=0$, and
$l_1=2n-1$, a contradiction to $l_1 \in [0, 2n-2]$.

\smallskip
\noindent
CASE 1.2: \,$l_6=0$.

If $l_8+l_{10}=2$, then $l_5=l_7=l_9=l_{11}=0$ and $l_1=n$, a
contradiction to $(2)$.

Suppose that $l_8+l_{10}=0$. Then $l_4=n-1+l_{11}$, $l_3=n-1+l_7$,
and $l_1= (n-1)(1+l_7+l_{11})+l_5+l_9$. If $l_7+l_{11}=1$, then
$l_1=2n-2+l_5+l_9$ and hence $l_1= 2n-2$, a contradiction to $(1)$.
If $l_7 + l_{11}=0$, then $l_5=l_9=1$ and $l_1=n+1$, a contradiction
to $(2)$. If $l_7 + l_{11}=2$, then $l_5=l_9=0$ and $l_1=n-3$, a
contradiction to $(2)$.

Suppose that $l_8+l_{10}=1$. Then $l_3 \equiv 2n-1+l_7 \mod 2n$,
which implies $l_7=1$ and $l_3=0$. Then $l_1 \equiv
2n-2+l_5+l_6+l_8+l_9 \mod 2n$, which implies $l_8=0$, $l_{10}=1$, and
$l_1 = 2n-2$, a contradiction to $(2)$.

\bigskip
\noindent
CASE 2: \,$l_5 + \ldots + l_{11} = 4$ and $l_2=n-2$.

\smallskip
\noindent
CASE 2.1: \,$l_6=1$.

If $l_8+l_{11}=1$, then $l_4 = 2n-1$, a contradiction to $l_4 \in
[0, 2n-2]$.

Suppose that $l_8+l_{11}=0$. If $l_7=1$, then $l_1 \equiv
2n-2+l_5+l_9 \mod 2n$. Since $l_1 \in [0, 2n-2]$ and $l_5 + \ldots +
l_{11} = 4$, it follows that $l_5=l_9=1$ and $l_1=0$, a
contradiction to $(2)$. If $l_7=0$, then $l_5=l_6=l_9=l_{10}=1$ and
$l_3 \equiv 2n-1 \mod 2n$, a contradiction to $l_3 \in [0, 2n-2]$.

Suppose that $l_8+l_{11}=2$. If $l_7=1$, then $l_5=l_9=l_{10}=0$ and
$l_1 = n-2$, a contradiction to $(2)$. If $l_7 =0$, then $l_3 \equiv
n-1 + n(1+l_{10}) \mod 2n$ and thus $l_{10}=1$, $l_5=l_9=0$, and
$l_1 \equiv 2n-1 \mod 2n$, a contradiction to $l_1 \in [0, 2n-2]$.

\smallskip
\noindent
CASE 2.2: \,$l_6=0$.

If $l_8+l_{10}=0$, then $l_5=l_7=l_9=l_{11}=1$ and $l_1=n-2$, a
contradiction to $(2)$.

Suppose that $l_8+l_{10}=1$. If $l_7=1$, then $l_3 \equiv 2n-1 \mod
2n$, a contradiction to $l_3 \in [0, 2n-2]$.  If $l_7=0$, then
$l_5=l_9=l_{11}=1$ and $l_1 \equiv 2n-1 + l_8 \mod 2n$, which implies
that $l_8=1$, $l_{10}=0$, and $l_1=0$, a contradiction to $(2)$.

Suppose that $l_8+l_{10}=2$. If $l_7+l_{11} = 0$, then $l_5=l_9=1$
and $l_1=n+1$, a contradiction to $(2)$. If $l_7+l_{11}=2$, then
$l_5=l_9=0$ and $l_1=n-3$, a contradiction to $(2)$. If
$l_7+l_{11}=1$, then $l_5+l_9 =1$ and $l_1 \equiv 2n-1 \mod 2n$, a
contradiction to $l_1 \in [0, 2n-2]$.

\bigskip
\noindent
CASE 3: \,$l_5 + \ldots + l_{11} = 6$ and $l_2=n-3$.

\smallskip
If $0 \in \{l_5, l_7, l_8, l_9, l_{10} \}$, then $l_4 \equiv 2n-1
\mod 2n$, a contradiction to $l_4 \in [0, 2n-2]$. If $l_6=0$, then
$l_1 = n-2$, a contradiction to $(2)$. If $l_{11} = 0$, then
$l_1=0$, a contradiction to $(2)$.

\bigskip
\noindent
CASE 4: \,$l_5 + \ldots + l_{11} = 6$ and $l_2 = 2n-3$.

\smallskip
If $l_5=0$ or $l_{11}=0$, then $l_3 \equiv 2n-1 \mod 2n$, a
contradiction to $l_3 \in [0, 2n-2]$. If $l_6=0$, then $l_4 \equiv
2n-1 \mod 2n$, a contradiction to $l_4 \in [0, 2n-2]$. If
$l_{10}=0$, then $l_1 \equiv 2n-1 \mod 2n$, a contradiction to $l_1
\in [0, 2n-2]$. If $l_7=0$, then $l_1=n$; if $l_8=0$, then
$l_1=2n-2$; if $l_9=0$, then $l_1=2n-2$. All these three cases give
a contradiction to $(2)$.
\end{proof}

\medskip
In two recent papers, B. Girard states a conjecture on the structure of
extremal zero-sum free sequences.  We recall the required
terminology.

Let  $G = C_{q_1} \oplus \ldots \oplus C_{q_s}$ be the direct
decomposition of the group $G$ into cyclic groups of prime power order, where
$s = \mathsf r^* (G) = \sum_{p \in \mathbb P} \mathsf r_p (G)$ is
the total rank of $G$, and set
\[
\mathsf k^* (G) = \sum_{i=1}^s \frac{q_i-1}{q_i} \,, \quad
\text{with} \quad \mathsf k^* (G) = 0 \ \text{for} \ G \
\text{trivial} \,.
\]
For a sequence $S = g_1 \cdot \ldots \cdot g_l$ over $G$,
\[
\mathsf k (S) = \sum_{i=1}^l \frac{1}{\ord (g_i)}
\]
denotes its {\it cross number}, and
\[
\mathsf k (G) = \max \{ \mathsf k (U) \mid U \in \mathcal F (G) \
\text{zero-sum free} \} \in \mathbb Q
\]
is the {\it little cross number} of $G$. If $(e_1, \ldots, e_s)$ is
a basis of $G$ with $\ord (e_i) = q_i$ for all $i \in [1, s]$, then
$S = \prod_{i=1}^s e_i^{q_i-1}$ is zero-sum free and hence $\mathsf
k^* (G) = \mathsf k (S) \le \mathsf k (G)$. Equality holds in
particular for $p$-groups, and there is no known group $H$ with
$\mathsf k^* (H) < \mathsf k (H)$. We refer to \cite[Chapter
5]{Ge-HK06a} for more information on the cross number and to
\cite{Gi09a, Ge-Gr09c} for recent progress. Now we formulate the
conjecture of B. Girard (see \cite[Conjecture 1.2]{Gi08b} and
\cite[Conjecture 2.1]{Gi10a}).

\medskip
\begin{conjecture}[B. Girard]
If $G \cong C_{n_1} \oplus
\ldots \oplus C_{n_r}$ with $1 < n_1 \t \ldots \t n_r$ and $S \in
\mathcal F (G)$ is zero-sum free with $|S| \ge \mathsf d^* (G)$,
then
\[
\mathsf k (S) \le \sum_{i=1}^r \frac{n_i-1}{n_i} \,.
\]
\end{conjecture}
The conjecture holds true for cyclic groups, $p$-groups (see
\cite[Proposition 2.3]{Gi08b}) and for groups of rank two (this
follows from the characterization of all minimal zero-sum sequences
of maximal length, \cite{Sc10b, Ga-Ge-Gr10a}). Suppose that  $G =
C_n^r$. If true, the conjecture would imply that $\mathsf d (G) =
\mathsf d^* (G)$ and, moreover, that all elements occurring in a
zero-sum free sequence of length $\mathsf d^* (G)$ have maximal
order $n$ (\cite[Proposition 2.1]{Gi08b}).

\medskip
\begin{corollary} \label{3.2}
Let $G = C_{2n}^r$ with $n \ge 3$ odd and $r \ge 5$. Then there
exists a zero-sum free sequence $T \in \mathcal F (G)$ and an
element $g \in G$ with $\ord (g) = n$ such that
\[
\mathsf v_g (T) = n-1 \,, \ |T| = \mathsf d^* (G) - (n-2) \quad
\text{and} \quad \mathsf k (T) = r \frac{2n-1}{2n} + \frac{1}{2n}
\,.
\]
In particular, if  $n=3$ and $r=5$, then  $|T| = \mathsf d^* (G)-1$,  $\mathsf k
(T)
> r(2n-1)/(2n)$, and there is no zero-sum free  sequence $T^* \in
\mathcal F (G)$ such that $T^* = g_1g_2T'$ and $T = (g_1+g_2)T'$,
where $g_1, g_2 \in G$ and $T' \in \mathcal F (G)$.
\end{corollary}

\begin{proof}
Let $(e_1',e_2, \ldots, e_r)$ be a basis of $G$ with $\ord (e_1') =
\ord (e_2) = \ldots = \ord (e_r) = 2n$. Let $e_1 = n e_1'$ and $S^*
\in \mathcal F ( \langle e_1, e_2, e_3, e_4, e_5 \rangle)$ be as
constructed in the proof of Theorem \ref{3.1}. Then
\[
|S^*| = 8n-2 \quad \text{and} \quad \mathsf k (S^*) = \frac{|S|}{2n}
\,.
\]
Since $\ord (2 e_1') = n$ and $\langle 2e_1', e_6, \ldots,
e_r\rangle \cap \langle e_1, \ldots, e_5\rangle = \{0\}$, the
sequence
\[
T = (2e_1')^{n-1} S^* \prod_{i=6}^r e_i^{2n-1}
\]
is zero-sum free and has the required properties.\\
In the case $n=3$ and
$r=5$, we have checked numerically---by a variant of the SEA (see Algorithm \ref{SEA}) with reduced search depth---that there is no such sequence
$T^*$, and the remaining assertions follow from the general case of the corollary.
\end{proof}

\medskip
\begin{remark} \label{3.3}
Thus, for the group $G = C_6^5$, the sequence $T$ given in Corollary
\ref{3.2} shows that the Conjecture is sharp, in the sense that the
assumption $|S| \ge \mathsf d^* (G)$ cannot be weakened to $|S| \ge
\mathsf d^* (G)-1$. But it shows much more.

Suppose that $G$ is cyclic of order $|G| = n \ge 3$. A simple
argument shows that $\mathsf d (G) = \mathsf d^* (G) = n-1$ and
every zero-sum free sequence $S$ of length $|S| = n-1$ has the form
$S = g^{n-1}$ for some $g \in G$ with $\ord (g) = n$. It was a
well-investigated problem in Combinatorial Number Theory to extend
this structural result to shorter zero-sum free sequences. In 2007,
S. Savchev and F. Chen could finally show that, for every zero-sum
free sequence $S$ of length $|S|
> (n+1)/2$, there is a $g \in G$ such that $S = (n_1g) \cdot \ldots \cdot (n_lg)$,
where \ $l = |S| \in
      \N$, $1 = n_1 \le \ldots \le n_l$, $n_1 + \ldots + n_l = m < \ord (g)$
      and $\Sigma (S) = \{g, 2g, \ldots , m g \}$ (see
      \cite{Sa-Ch07a} and \cite[Theorem 5.1.8]{Ge09a}).
Thus $S$ is obtained by taking some factorization $(g^{n_1})\cdot\ldots\cdot(g^{n_l})=g^{m-1}$ of the sequence $g^{m-1}$ and replacing each $g^{n_i}$ by $\sigma(g^{n_i})=n_ig$ for $i \in [1,l]$.
By Corollary \ref{3.2}, such a result does not hold for $C_6^5$, not even for zero-sum free sequences of length $\mathsf d^*(G)-1$.
\end{remark}

\bigskip
\section{Description of the Computational Approach} \label{4}
\bigskip

Computational methods have already been used successfully for a
variety of zero-sum problems (see recent work of G. Bhowmik, Y.
Edel, C. Elsholtz, I. Halupczok, J.-C. Schlage-Puchta et al.
\cite{El04, Ed08a,E-E-G-K-R07, Bh-Ha-SP10a}). Inspired by former
work in the groups $C_2^2 \oplus C_{2n}^3$ for $n \ge 3$ odd, we
found many examples of zero-sum free sequences $S$ over $G = C_2
\oplus C_6^4$ of length $|S| = \mathsf d ^* (G) + 1$. These were
used as starting points in a computer based search in the group $C_2
\oplus C_{10}^4$, which will be explained in detail below.

\begin{algorithm}
\caption{Sequence Extension Algorithm: $(g_1,g_2,g_3,g_4,g_5,g_6) \leftarrow \mathrm{SEA}(S)$}
\label{SEA}
\begin{algorithmic}
\STATE $\sigma_0 \leftarrow -\Sigma(S) \cup \lbrace0\rbrace$
\FORALL{$g_1\in G$ such that $g_1\notin \sigma_0$}
\STATE $\sigma_1 \leftarrow \emptyset, \, \sigma_2 \leftarrow \emptyset, \, \sigma_3 \leftarrow \emptyset, \, \sigma_4 \leftarrow \emptyset, \, \sigma_5 \leftarrow \emptyset$
\STATE $G_1 \leftarrow \emptyset, \, G_2 \leftarrow \emptyset, \, G_3 \leftarrow \emptyset, \, G_4 \leftarrow \emptyset, \, G_5 \leftarrow \emptyset$
\STATE $\sigma_1 \leftarrow \sigma_0 \cup (\sigma_0 - g_1)$
\FORALL{$g \in G$ such that $g \notin \sigma_1$}
\STATE $G_1 \leftarrow G_1 \cup \lbrace g \rbrace$
\ENDFOR
\FORALL{$g_2 \in G_1$ such that $g_2 \leq g_1$}
\STATE $(G_2,\,\sigma_2) \leftarrow \mathrm{SCA}( G_1, \sigma_1, g_2 ) $
\FORALL{$g_3 \in G_2$ such that $g_3 \leq g_2$}
\STATE $(G_3,\,\sigma_3) \leftarrow \mathrm{SCA}( G_2, \sigma_2, g_3 ) $
\FORALL{$g_4 \in G_3$ such that $g_4 \leq g_3$}
\STATE $(G_4,\,\sigma_3) \leftarrow \mathrm{SCA}( G_3, \sigma_3, g_4 ) $
\FORALL{$g_5 \in G_4$ such that $g_5 \leq g_4$}
\STATE $(G_5,\,\sigma_4) \leftarrow \mathrm{SCA}( G_4, \sigma_4, g_5 ) $
\FORALL{$g_6 \in G_5$ such that $g_6 \leq g_5$}
\RETURN $(g_1,g_2,g_3,g_4,g_5,g_6)$
\ENDFOR
\ENDFOR
\ENDFOR
\ENDFOR
\ENDFOR
\ENDFOR
\end{algorithmic}
\end{algorithm}

\begin{algorithm}
\caption{Sumset Computation Algorithm: $(G', \sigma') \leftarrow \mathrm{SCA}(G, \sigma, g)$}
\label{SCA}
\begin{algorithmic}
\STATE $\sigma' \leftarrow \sigma$
\STATE $G' \leftarrow \emptyset$
\FORALL{$h \in G$}
\IF{$(g+h) \in \sigma$}
\STATE $\sigma' \leftarrow \sigma' \cup \lbrace h \rbrace$
\ELSE
\STATE $G' \leftarrow G' \cup \lbrace h \rbrace$
\ENDIF
\ENDFOR
\RETURN $(G', \sigma')$
\end{algorithmic}
\end{algorithm}

The Sequence Extension Algorithm (SEA) (see Algorithm~\ref{SEA}) uses a smart brute force approach, where the computation time is significantly reduced by algorithmic short-cuts, efficient data structures for set testing, and fast look-up tables for group operations. The program was implemented in the C/C++ programming language.
Furthermore, MPI parallelization was used to enable the execution of the program on cluster computers and supercomputers with thousands of computing cores. The parallelization scheme is a simple master-slave algorithm, where the master thread partitions the outermost loop over all group elements and sends out these work items to the available pool of slave threads. In this scheduler, a dynamic policy with chunk size one is used; that is, the master thread sends out only one work item to the next slave thread available. Although this leads to some communication overhead between the master and the slave threads, it is quite reasonable as the necessary computation time for one work item can vary by a factor of more than $25000$, i.e., from less than a second up to a few hours.
The first major algorithmic short-cut is restricting the search to ascending sequences with respect to coordinates in a basis and lexicographic ordering, thus omitting all permutations arising from the same sequence.
The second short-cut is keeping track of all group elements not in the set of negative subsums in additional vectors---namely $G_1,\,G_2,\,G_3,\,G_4$, and $G_5$ in the SEA (see Algorithm~\ref{SEA}). These vectors are used to massively speed up the Sumset Computation Algorithm (SCA) (see Algorithm~\ref{SCA}) by avoiding many unnecessary tests. Typically, the vectors $G_i$, for $i\in[1,5]$, consist of only a few hundred group elements while $\#G=20000$---this means a speed up by a factor of about $20$ to $200$ in each step of the descending inner loops in the SCA (see Algorithm~\ref{SCA}). As a last step of optimization, we pre-compute a look-up table for subtraction in $G$, which is stored in a very specific way such that we can use it for the tests in the SCA (see Algorithm~\ref{SCA}) and benefit from data caching and pre-fetching on modern CPUs while accessing the elements in a single line of the look-up table.

\begin{table}[ht]
\centering
\begin{tabular}{|c|c c c c|c|}
\hline
Test Set & Test Sequences & Complete & Hits & Extensions &Compute Time\\[0.5ex]
\hline\hline
a & 81 & 27 & 0 & 0 & 28,366\\
b & 81 & 52 & 5 & 92 & 26,670\\
c & 81 & 52 & 5 & 252 & 26,688\\
d & 81 & 75 & 4 & 196 & 15,808\\
\hline
 & 324 & 206 & 14 & 540 & \textbf{97,534}\\
\hline
\end{tabular}
\caption{Statistics of the four test runs $a$, $b$, $c$, and $d$ on the \emph{cineca} supercomputer. The compute time is given in hours w.r.t. a single IBM Power6 4.7 GHz SMT CPU core.}
\label{TestSet}
\end{table}

The computations for the test sequences $a$, $b$, $c$, and $d$ on the \emph{cineca} supercomputer used 64 threads with a single master and 63 slaves. The parallel efficiency of the algorithm, due to the independent nature of the computations, proved to be very good. The \emph{cineca} supercomputer is an IBM pSeries 575 Infiniband cluster with 168 computing nodes and 5376 computing cores. Every node has eight IBM Power6 4.7 GHz quad-core CPUs with simultaneous multi-threading (SMT) and 128 GB of shared memory.
Performance tests of the parallel algorithm showed that the best configuration to run a single work item on is a single computing node with 64 threads with SMT enabled. Single node work loads were also scheduled typically within a day on the \emph{cineca} supercomputer. The complete run of all test sets $a$, $b$, $c$, and $d$ took about a week on the \emph{cineca} supercomputer with an equivalent of nearly 100,000 SMT CPU core hours computation time. The run time of a work item on a single computing node was limited to six hours wall clock time by batch processing system policy. Nevertheless, most work items finished within these time restrictions, namely, 206 out of 324, and the ones that did not finish had most of the time only very few elements left to check, so we decided not to reschedule these work items for completion. The full statistics of the computations is given in Table~\ref{TestSet}.

\bigskip


\providecommand{\bysame}{\leavevmode\hbox to3em{\hrulefill}\thinspace}
\providecommand{\MR}{\relax\ifhmode\unskip\space\fi MR }
\providecommand{\MRhref}[2]{%
  \href{http://www.ams.org/mathscinet-getitem?mr=#1}{#2}
}
\providecommand{\href}[2]{#2}

\end{document}